\documentclass[letterpaper,10pt]{IEEEtran}

\usepackage{graphicx}
\usepackage{epstopdf}
\usepackage{mathbbol}
\usepackage{mathrsfs}
\usepackage{url}
\usepackage{amsmath,amsfonts, amssymb,color}
\usepackage{amsthm}
\usepackage[width=.85\textwidth, font=footnotesize]{caption}
\newcommand{\ignore}[1]{}

\newtheorem{theorem}{Theorem}[section]
\newtheorem{lemma}[theorem]{Lemma}
\newtheorem{remark}[theorem]{Remark}
\newtheorem{proposition}[theorem]{Proposition}
\newtheorem{assumption}[theorem]{Assumption}
\newtheorem{definition}[theorem]{Definition}

\makeatletter
\newcommand{\removelatexerror}{\let\@latex@error\@gobble}
\makeatother
\usepackage[T1]{fontenc}
\usepackage[utf8]{inputenc}

\usepackage[english]{babel}
\usepackage[usenames,dvipsnames,table]{xcolor}
\usepackage{amsmath,amsthm,amssymb}
\usepackage{graphicx}
\usepackage{url}
\usepackage{float}
\usepackage{tikz}
\usetikzlibrary{shapes,arrows,positioning}
\usepackage{empheq}

\usepackage{pgfplots} 		       
\usepackage[utf8]{inputenc}
\usepackage{dsfont}   		       
 \usepackage[siunitx]{circuitikz} 
\usepackage[mathscr]{euscript}
\usepackage{mathtools}
\usepackage{xargs} 

\usepackage[ruled,shortend,commentsnumbered,linesnumbered]{algorithm2e}

\usepackage{etoolbox} \let\bbordermatrix\bordermatrix
\patchcmd{\bbordermatrix}{8.75}{4.75}{}{}
\patchcmd{\bbordermatrix}{\left(}{\left[}{}{}
    \patchcmd{\bbordermatrix}{\right)}{\right]}{}{}

\graphicspath{{./img/}{./}} 
\newcommand{\argmin}{\operatorname{argmin}}

\newcommand{\real}{\mathbb{R}}
\newcommand{\realextended}{\overline{\real}}

\newcommand{\naturalnumbers}{\mathbb{N}}
\newcommand{\norm}[1]{\ensuremath{\| #1 \|}}
\newcommand{\until}[1]{[#1]}
\newcommand{\map}[3]{#1:#2 \rightarrow #3}
\newcommand{\setdef}[2]{\{#1 \; | \; #2\}}
\newcommand{\setdefb}[2]{\Bigl\{#1 \; \Big| \; #2\Bigr\}}

\newcommand{\MM}{\mathcal{M}}
\newcommand{\abs}[1]{\ensuremath{\left\lvert{#1}\right\rvert}}
\newcommand{\cl}{\operatorname{cl}}
\newcommand{\lm}{\lambda}
\newcommand{\setr}[1]{\{#1\}}

\newcommand{\xo}{x^\star}

\newcommand{\lmo}{\lm^\star}
\newcommand{\topt}{t^\star}

\newcommand{\yt}{\tilde{y}}
\newcommand{\st}{\operatorname{s.} \operatorname{t.}} 

\newcommand{\Eb}{\mathbb{E}}
\newcommand{\Pb}{\mathbb{P}}
\newcommand{\Rb}{\mathbb{R}}

\newcommand{\BB}{\mathcal{B}}
\newcommand{\HH}{\mathcal{H}}

\newcommand{\PP}{\mathcal{P}}
\renewcommand{\SS}{\mathcal{S}}

\newcommand{\CVaR}[1]{\operatorname{CVaR}^{#1}}

\newcommand{\Pbhat}{\widehat{\Pb}}
\newcommand{\data}{\widehat{\xi}}

\newcommand{\Xdcp}{\widehat{X}_{\mathtt{DCP}}}
\newcommand{\Xcdcpin}{\widehat{X}^{\mathtt{in}}_{\mathtt{CDCP}}}
\newcommand{\Xcdcp}{\widehat{X}_{\mathtt{CDCP}}}
\newcommand{\Xsa}[1]{\widehat{X}_{\mathtt{SA},#1}}
\newcommand{\Xscp}[1]{\widehat{X}_{\mathtt{SCP},#1}}

\newcommand{\oprocendsymbol}{\hbox{$\bullet$}}
\newcommand{\oprocend}{\relax\ifmmode\else\unskip\hfill\fi\oprocendsymbol}
\newcommand{\longthmtitle}[1]{\mbox{}\textup{\textsl{(#1):}}}

\allowdisplaybreaks

\usepackage[normalem]{ulem} 
\definecolor{new}{rgb}{0.55,0,0.55}

\usepackage[prependcaption,colorinlistoftodos]{todonotes}

\renewcommand{\theenumi}{(\roman{enumi})}

\allowdisplaybreaks

\makeatletter
\newcommand{\thickhline}{%
  \noalign {\ifnum 0=`}\fi \hrule height 1pt
  \futurelet \reserved@a \@xhline
}
\newcolumntype{"}{@{\hskip\tabcolsep\vrule width 1pt\hskip\tabcolsep}}
\makeatother

\title{Data-Driven Chance Constrained Optimization \\ under Wasserstein Ambiguity Sets}

\author{Ashish R. Hota, Ashish Cherukuri and John Lygeros   
\thanks{The authors are with the Automatic Control Laboratory, ETH Z\"{u}rich, Switzerland. Email: {\{ahota,cashish,lygeros\}@control.ee.ethz.ch}.}}

\begin{document}

\maketitle
\thispagestyle{empty}
\pagestyle{empty}

\begin{abstract}
We present a data-driven approach for distributionally robust chance constrained optimization problems (DRCCPs). We consider the case where the decision maker has access to a finite number of samples or realizations of the uncertainty. The chance constraint is then required to hold for all distributions that are close to the empirical distribution constructed from the samples (where the distance between two distributions is defined via the Wasserstein metric). We first reformulate DRCCPs under data-driven Wasserstein ambiguity sets and a general class of constraint functions. When the feasibility set of the chance constraint program is replaced by its convex inner approximation, we present a convex reformulation of the program and show its tractability when the constraint function is affine in both the decision variable and the uncertainty. For constraint functions concave in the uncertainty, we show that a cutting-surface algorithm converges to an approximate solution of the convex inner approximation of DRCCPs. Finally, for constraint functions convex in the uncertainty, we compare the feasibility set with other sample-based approaches for chance constrained programs.
\end{abstract}

\section{Introduction}

Numerous engineering applications encounter optimization problems with constraints dependent on uncertain parameters. Solution methodologies for such optimization problems fall broadly into two categories. In {\it robust optimization}, the aim is to take a decision that is feasible for all realizations of the uncertainty \cite{ben2009robust}. This approach often yields conservative solutions with regard to the optimal value and requires the support of the uncertainty to be bounded and known to the decision maker. In contrast, a {\it chance constrained program} (CCP) has soft probabilistic constraints on the decision variable in place of the hard ones present in a robust optimization \cite{shapiro2009lectures}; specifically, the aim is to compute a solution that satisfies the constraint with high probability. CCPs are increasingly used in many applications, such as stochastic model predictive control \cite{farina2016stochastic,schildbach2014scenario}, robotics \cite{blackmore2011chance,vitus2016stochastic}, energy systems \cite{vrakopoulou2017chance,guo2018data} and autonomous driving \cite{carvalho2015automated}.

In order to solve a CCP, the decision maker needs to know the probability distribution of uncertain parameters. In practice, this information is often unavailable and instead, the decision maker has access to data about the uncertainty in the form of samples. {\it Scenario}~\cite{calafiore2005uncertain,campi2008betterbound,calafiore2010random} and {\it sample approximation}~\cite{luedtke2008sample} approaches use this data to compute an approximate solution of the CCP. In the scenario approach, the constraint involving uncertainty is required to hold for every available sample, while in the sample and discard scenario approach \cite{campi2011samplediscard} and the sample approximation approach \cite{luedtke2008sample}, it is required to hold for a large fraction of samples. Their main advantage is that if the samples are drawn from a true underlying distribution and the number of samples is sufficiently large, the solutions are feasible for the original CCP with high probability. However, in practice, samples may be few and not be drawn from the true distribution. In such settings, it is desirable to find a solution that satisfies the chance constraint for all distributions that belong to a suitably defined family of distributions, or a so-called {\it ambiguity set}. This class of problems is known as {\it distributionally robust chance constrained programs} (DRCCPs) and is the focus of this paper.

In distributionally robust stochastic optimization (DRSO) in general and DRCCPs in particular, the ambiguity set is defined either as a set of probability distributions that satisfy certain moment constraints \cite{delage2010distributionally,zymler2013joint,hanasusanto2017ambiguous} or that are close under an appropriate distance function, such as the Prokhorov metric \cite{erdougan2006ambiguous} or $\phi$-divergence \cite{jiang2016data-driven}. DRCCPs with moment based ambiguity sets were recently considered for designing controllers for stochastic systems \cite{van2016distributionally} and to solve optimal power flow problems with uncertain renewable energy generation \cite{zhang2017distributionally}. Recent work in DRSO has shown that ambiguity sets based on Wasserstein distance \cite{villani2003topics} have desirable out-of-sample performance and asymptotic guarantees \cite{gao2016wasserstein,peyman2017wasserstein}. DRSO with Wasserstein ambiguity sets were recently applied in optimal power flow problems \cite{guo2018data} and uncertain Markov decision processes \cite{yang2017convex}. Motivated by these attractive features, we consider a data-driven approach for DRCCPs where the ambiguity set is defined as the set of distributions that are close (in the Wasserstein distance) to the empirical distribution induced by the observed samples (see Section \ref{sec:prelims} for a formal definition). 

The literature on DRCCPs with Wasserstein ambiguity sets is limited. The authors in \cite{ahmed2018bicriteria} first showed that it is strongly NP-Hard to solve a DRCCP with Wasserstein ambiguity sets and proposed a bi-criteria approximation scheme for covering constraints. While preparing this paper, we became aware of two recent working papers that presented reformulations and approximations of DRCCPs under Wasserstein ambiguity sets \cite{xie2018drccp,kuhn2018drccp} and for constraint functions that are affine in both the decision variable and the uncertainty. Both \cite{xie2018drccp,kuhn2018drccp} show that the exact feasibility set of DRCCPs with affine constraints can be reformulated as mixed integer conic programs. Specifically, Xie \cite{xie2018drccp} studies individual chance constraints and joint chance constraints with right hand side uncertainty, while Chen et. al., \cite{kuhn2018drccp} consider general affine joint chance constraints. Both papers appeared subsequent to the appearance of a preliminary version of our work. In this paper, we lay the foundations for tractable computation of (approximate) solutions of DRCCPs under data-driven Wasserstein ambiguity sets for a broad class of constraint functions.

\noindent {\bf Summary of contributions:} We first reformulate DRCCPs under Wasserstein ambiguity sets under general continuity and boundedness assumptions on the constraint functions (as opposed to the affine case studied in \cite{xie2018drccp,kuhn2018drccp}). We then focus on developing tractable reformulations and algorithms for DRCCPs. Since the feasibility set of (DR)CCPs is nonconvex except for restrictive special cases \cite{prekopa1970probabilistic}, we consider constraint functions that are convex in the decision variable, and replace the exact feasibility set of the DRCCP with its convex conditional value-at-risk (CVaR) approximation following \cite{nemirovski2006convex} leading to a convex program that approximates the original DRCCP. We then present a tractable reformulation of the CVaR approximation when the constraint function is the maximum of functions that are affine in both the decision variable and the uncertainty, and the support of the uncertainty is a polyhedron. When the constraint function is concave in the uncertainty, we show that a recently developed central cutting-surface algorithm for semi-infinite programs \cite{mehrotra2014cutting,luo2017decomposition} can be used to compute an approximately optimal solution of the CVaR approximation of the DRCCP. Finally, when the constraint function is convex in the uncertainty, we compare the feasibility set of the CVaR approximation with those of the sample approximation approach \cite{luedtke2008sample} and the scenario approach \cite{calafiore2005uncertain,campi2008betterbound}.

\noindent {\bf Notation:} The sets of real, positive real, non-negative real, and natural numbers are denoted by $\Rb$, $\Rb_{>0}$, $\Rb_{\ge 0}$, and $\naturalnumbers$, respectively. The extended reals are $\realextended = \real \cup \{+ \infty, - \infty \}$. For $N \in \naturalnumbers$, we let $[N] := \{1,2,\ldots,N\}$. For brevity, we denote $\max(x,0)$ by $x_+$. The closure of a set $\SS$ is denoted by $\cl(\SS)$. Feasibility sets constructed using data are denoted by $\widehat{\cdot}$. For a set $\SS$ and $N \in \naturalnumbers$, we denote the $N$-fold cartesian product as $\SS^N := \Pi_{i=1}^N \SS$. Similar notation holds for the $N$-fold product of any probability distribution. 
\section{Technical Preliminaries}\label{sec:prelims}

Here we collect preliminary notions and results on CCPs, conditional value-at-risk, and Wasserstein ambiguity sets.

\subsection{Chance Constrained Programs and CVaR Approximation}\label{subsec:cc-approx}

Throughout we consider $\Xi$ to be a complete separable metric space with metric $d$. Let $\BB({\Xi})$ and $\PP(\Xi)$ be the Borel $\sigma$-algebra and the set of Borel probability measures on $\Xi$, respectively. A canonical CCP is of the form
\begin{equation}\label{eq:def_ccp}
\begin{aligned}
\underset{x\in X}{\min} & \quad c^\intercal x
\\
\st  & \quad  \Pb(F(x,\xi) \leq 0) \geq 1-\alpha, 
\end{aligned}
\end{equation}
where $X \subseteq \Rb^n$ is a closed convex set, $c \in \Rb^n$, $\alpha \in (0,1)$, $\Pb \in \PP(\Xi)$, and $\map{F}{\Rb^n \times \Xi}{\Rb}$. With the exception of a restricted class of distributions and constraint functions, the feasibility set of \eqref{eq:def_ccp} is nonconvex even when $X$ is convex and $F$ is convex in $x$ for every $\xi$ \cite{prekopa1970probabilistic}. 

Several convex approximations exist that overcome this intractability. We now describe the approximation framework developed in~\cite{nemirovski2006convex} that plays a central role in our results. Consider the function $\map{\psi(z)}{\real}{\real}$, given as $\psi(z) = \max(z+1,0)$. This function belongs to the class of moment generating functions defined in~\cite{nemirovski2006convex}. For a given $\Pb \in \PP(\Xi)$, define $\map{\Psi_{\Pb}}{\real^n \times \real}{\real}$ as
\begin{equation}\label{eq:mgfC_eq}
\Psi_{\mathbb{P}}(x,t) := t \Eb_{\Pb}[\psi(t^{-1}F(x,\xi))].
\end{equation}
Note that if $x \mapsto F(x,\xi)$ is convex for every $\xi \in \Xi$, then
$\Psi_{\Pb}$ is convex in $x$ and $t$. Furthermore, we have
\begin{equation}\label{eq:def_cca_ccp}
\inf_{t > 0} [\Psi_{\mathbb{P}}(x,t)-t\alpha] \leq 0 \implies
\mathbb{P}(F(x,\xi) \leq 0) \geq 1-\alpha.
\end{equation}
Therefore, replacing the chance constraint by $\inf_{t > 0} [\Psi_{\mathbb{P}}(x,t)-t\alpha] \leq 0$ gives a convex conservative approximation of the CCP \eqref{eq:def_ccp}.  This approximation is equivalent to replacing the probabilistic constraint with its conditional value-at-risk (CVaR). Formally,  the CVaR of a random variable $Z$ with distribution $\Pb$ at level $\alpha$ is~\cite{rockafellar2000optimization}
\begin{align}
\CVaR{\Pb}_{1-\alpha}(Z):= \inf_{t \in \Rb} \bigl[
\alpha^{-1} \Eb_{\Pb}[(Z+t)_+] -t \bigr]. \label{eq:CVaR-def}
\end{align}
One can show (as done in~\cite{nemirovski2006convex}) that 
\begin{align}
& \inf_{t > 0} [\Psi_{\Pb}(x,t)-t\alpha] \leq 0 \! 
\!  \iff \! \! \CVaR{\Pb}_{1-\alpha}(F(x,\xi)) \leq 0.
\label{eq:equi}
\end{align}

We note that condition \eqref{eq:equi} is stronger than simply requiring $F(x,\xi) \leq 0$ with probability at least $1-\alpha$ as in this case, $F(x,\cdot)$ could take arbitrarily large values for realizations of $\xi$ with measure at most $\alpha$. In contrast, \eqref{eq:equi} requires the expected value of $F(x,\cdot)$ for the worst possible realizations of $\xi$ with measure $\alpha$ to be at most zero. In other words, \eqref{eq:equi} prescribes a condition on the expected violation of the chance constraint. We henceforth refer the convex conservative approximation of CCP, that is,  probabilistic constraint in \eqref{eq:def_ccp} replaced by \eqref{eq:equi}, as its {\it CVaR approximation}. 

\subsection{Wasserstein ambiguity sets}\label{subsec:dro}

Let $\PP_p(\Xi) \subseteq \PP(\Xi)$ be the set of Borel probability
measures with finite $p$-th moment for $p \in [1,\infty)$. Recall that $d$ is the metric on $\Xi$. Following \cite{villani2003topics}, for $p \in [1,\infty)$, the $p$-Wasserstein distance between measures $\mu, \nu \in \PP_p(\Xi)$ is
\begin{equation}\label{eq:def_wasserstein}
  (W_p(\mu,\nu))^p := \min_{\gamma \in \HH(\mu,\nu)}
  \left\{\int_{\Xi \times \Xi} d^p(\xi,\omega) \gamma(d\xi,d\omega) \right\},
\end{equation}
where $\HH(\mu,\nu)$ is the set of all distributions on
$\Xi \times \Xi$ with marginals $\mu$ and $\nu$. The minimum in \eqref{eq:def_wasserstein} is attained because $d$ is lower semicontinuous \cite{gao2016wasserstein}.  

In this paper, we define the ambiguity set as the set of all distributions that are close to the empirical distribution induced by the observed samples. Specifically, let $\Pbhat_N := \frac{1}{N}\sum^N_{i=1} \delta_{\data_i}$ be the empirical distribution constructed from the observed samples $\{\data_i\}_{i \in [N]}$. We define the data-driven Wasserstein ambiguity set as
\begin{equation}\label{eq:wasserstein-set}
\MM^\theta_N := \{\mu \in \PP_p(\Xi)| W_p(\mu,\Pbhat_N) \leq
\theta\},
\end{equation}
which contains all distributions that are within a distance $\theta \geq 0$ of $\Pbhat_N$. We now present a duality theorem for distributionally robust stochastic optimization over Wasserstein ambiguity sets from \cite{gao2016wasserstein} that is central to proving our reformulations. Let $\map{H}{\Xi}{\Rb}$ and consider the following primal and dual problems
\begin{subequations}\label{eq:gao_primal_dual}
	\begin{align}
	v_P & :=
	\sup_{\mu \in \PP_p(\Xi)} \setdefb{ \int_\Xi H(\xi)
		\mu(d\xi)}{W_p(\mu,\Pbhat_N) \leq \theta}, \label{eq:gao_primal}
	\\
	v_D & := \inf_{\lambda \geq 0} \Bigl[ \lambda \theta^p
	+ \frac{1}{N} \sum^N_{i=1} \sup_{\xi \in \Xi} [H(\xi) - \lambda d^p(\xi,\data_i)] \Bigr]. \label{eq:gao_dual}
	\end{align}
\end{subequations}
\begin{theorem}\longthmtitle{Zero-duality gap~\cite{gao2016wasserstein}}\label{theorem:gao_duality}
Assume that $H$ is upper semicontinuous and either $\Xi$ is bounded, or there exists $\xi_0 \in \Xi$ such that 
	$$ \underset{d(\xi,\xi_0) \to \infty}{\lim \sup}
	\frac{H(\xi)-H(\xi_0)}{d^p(\xi,\xi_0)} < \infty. $$
	Then, the dual problem~\eqref{eq:gao_dual} always admits a minimizer
	$\lambda^*$ and $v_p = v_D < \infty$.
\end{theorem}

We conclude with the stochastic min-max theorem due to
\cite{shapiro2002minimax} which will be required in proving one of our key results.

\begin{theorem}\longthmtitle{Stochastic min-max equality}\label{thm:min-max-shapiro}
  Let $\MM$ be a nonempty (not necessarily convex) set of probability
  measures on $(\Xi,\mathcal{B}(\Xi))$ where $\Xi \subseteq \Rb^m$ and
  $\BB(\Xi)$ is the Borel $\sigma$-algebra. Assume that $\MM$ is
  weakly compact. Let $T \subseteq \Rb^n$ be a closed convex
  set. Consider a function $g:\Rb^n \times \Xi \to \Rb$. Assume that
  there exists a convex neighborhood $V$ of $T$ such that for all
  $t \in V$, the function $g(t,\cdot)$ is measurable, integrable with
  respect to all $\Pb \in \MM$, and
  $\sup_{\Pb \in \MM} \Eb_{\Pb} [g(t,\xi)] < \infty$. Further assume
  that $g(\cdot,\xi)$ is convex on $V$ for all $\xi \in \Xi$. Let
  $\bar{t} \in \argmin_{t \in T} \sup_{\Pb \in \mathcal{M}}
  \mathbb{E}_{\Pb}[g(t,\xi)]$. Assume that for every $t$ in a
  neighborhood of $\bar{t}$, the function $g(t,\cdot)$ is bounded and
  upper semicontinuous on $\Xi$ and the function $g(\bar{t},\cdot)$ is
  bounded and continuous on $\Xi$. Then,
  \begin{equation*}
    \inf_{t \in T} \sup_{\Pb \in \MM} \Eb_{\Pb}[g(t,\xi)]
    = \sup_{\Pb \in \MM} \inf_{t \in T} \Eb_{\Pb}[g(t,\xi)].
  \end{equation*}
\end{theorem}

Note that the above theorem requires the ambiguity set to be weakly compact. This is indeed the case for Wasserstein ambiguity sets constructed from data as stated below.

\begin{proposition}[Corollary 2, \cite{pichler2017quantitative}]\label{proposition:wasserstein_compact}
The Wasserstein ambiguity set $\mathcal{M}^\theta_N$ is tight and weakly-compact.
\end{proposition}

We now start by presenting exact reformulations of DRCCPs with data-driven ambiguity set $\MM^\theta_N$. 
\section{Distributionally Robust Chance Constrained Program and Exact Reformulation}\label{sec:reformulation}

In this section, we describe our problem of interest: distributionally robust chance constrained program (DRCCP) with Wasserstein ambiguity sets. Following that, we present two exact reformulations of the DRCCP that have simpler representations. Let $\setr{\data_i}_{i=1}^N$ be a set of $N$ samples of $\xi$ available to the decision maker. Given this data and $\theta > 0$, the DRCCP for the Wasserstein ambiguity sets \eqref{eq:wasserstein-set} is 
 \begin{equation}\label{eq:def_drccp_dd}
 \begin{aligned}
 & {\min}\{c^\intercal x: x\in \widehat{X}_{\mathtt{DCP}}\}, \qquad \text{where} \\
 & \Xdcp := \setdefb{x \in X}{\underset{\Pb \in
 		\MM^\theta_N}{\sup} \Pb(F(x,\xi) > 0) \leq \alpha}.
\end{aligned}
\end{equation}
Note that if $F: \Rb^n \times \Xi \to \Rb^K$, then we can instead define F as the component-wise maximum of K constraints. We assume $F$ to be continuous. The probabilistic constraint defining $\Xdcp$ can be equivalently written as
$$ \underset{\Pb \in \MM^\theta_N}{\sup} \Pb(F(x,\xi) \! > \! 0) \! \leq \! \alpha  \! \iff \! \underset{\Pb \in \MM^\theta_N}{\inf} \Pb(F(x,\xi) \! \leq \! 0) \! \geq \! 1-\alpha  .$$

Note that~\eqref{eq:def_drccp_dd} involves optimization over a set of distributions. In order to get a handle on this infinite-dimensional optimization problem, we provide below exact reformulations that involve optimization over finite dimensions. The reformulations presented below were independently shown in \cite{xie2018drccp} for $F$ affine in both $x$ and $\xi$. Here we establish that the results hold more generally. 
\begin{theorem}\longthmtitle{Exact reformulations of DRCCP}\label{theorem:drccp_exact_reformulation}
	Let the function $\map{G}{\real^n \times \Xi}{\realextended}$ be given as  
	\begin{equation}\label{eq:dist_func}
	G(x,\data) := \begin{cases} \underset{\setdef{\xi}{F(x,\xi) > 0}}{\inf} d^p(\xi,\data), &  \setdef{\xi}{F(x,\xi) > 0} \not = \emptyset,
	\\
	+ \infty, &  \text{otherwise}.
	\end{cases}
	\end{equation}
	Suppose $\Xi = \real^m$ and there exists $\xi_0 \in \Xi$ such that
	\begin{equation}\label{eq:F-growth}
	\underset{d(\xi,\xi_0) \to \infty}{\lim \sup} \frac{F(x,\xi)-F(x,\xi_0)}{d^p(\xi,\xi_0)} < \infty, \quad \forall x \in X.
	\end{equation}
	Then, the feasibility set of the DRCCP~\eqref{eq:def_drccp_dd} satisfies 
	\begin{align}
	\Xdcp \!  = \! \left\{ \! \! \begin{array}{l}  x \in X \end{array} \! \! \Bigg| \! \! \begin{array}{l} \exists \lambda \geq 0, \lambda \theta^p + \frac{1}{N} \sum_{i=1}^N s_i \le\alpha, \\ s_i = \max\{1-\lambda G(x,\data_i), 0\}\end{array} \! \right\}. \label{eq:drccp_dist_exact}
	\end{align}
	In addition, if $\setdef{\xi}{ F(x,\xi) > 0}$ is nonempty for every $x \in X$, then
	\begin{align}
	\Xdcp = \left\{x \in X \Bigg| \frac{\theta^p}{\alpha} + \CVaR{\Pbhat_N}_{1-\alpha}(-G(x,\xi)) \leq 0\right\}. \label{eq:drccp_cvar_exact}
	\end{align}
\end{theorem}
\begin{proof}
We first show that $\Xdcp$ defined in~\eqref{eq:def_drccp_dd} is equivalent to the set in the right-hand side of~\eqref{eq:drccp_dist_exact}. We suppress the argument $x$ from the functions $F$ and $G$ as the arguments hold point wise for every $x \in X$. We evaluate
\begin{align}
& \sup_{\Pb \in \MM^\theta_N} \Pb(F(\xi)> 0 ) = \sup_{\Pb \in \MM^\theta_N} \Eb_\Pb[\mathbb{1}_{\cl(\xi: F(\xi) > 0)}] \notag
\\ & = \inf_{\lambda \geq 0} \lambda \theta^p + \frac{1}{N} \sum^N_{i=1}
      \sup_{\xi \in \Xi} [\mathbb{1}_{\cl(\xi: F(\xi) > 0)}-\lambda d^p(\xi,\data_i)], \label{eq:indicator-zero-duality}
\end{align}
where the first equality follows from~\cite[Proposition 4]{gao2016wasserstein}, and the second equality is a consequence of the strong duality theorem (Theorem~\ref{theorem:gao_duality}).\footnote{Recall that Theorem~\ref{theorem:gao_duality} requires the function within the expectation to be upper semicontinuous. Since the indicator function of an open set is lower semicontinuous, we replace it with its closure. This substitution is valid due to~\cite[Proposition 4]{gao2016wasserstein}.} Now let $\Xi_1 = \cl(\xi: F(\xi) > 0)$ and $\Xi_2 = \Xi \setminus \Xi_1$. For each term in the summation~\eqref{eq:indicator-zero-duality}, we introduce an auxiliary variable as
\begin{align*}
s_i & = \sup_{\xi \in \Xi} [\mathbb{1}_{\cl(\xi: F(\xi) > 0)}-\lambda d^p(\xi,\data_i)]
\\ & = \textstyle \max\{\sup_{\xi \in \Xi_1} [1-\lambda d^p(\xi,\data_i)], \sup_{\xi \in \Xi_2} -\lambda d^p(\xi,\data_i) \}
\\ & = \textstyle \max\{1-\lambda G(\data_i), \sup_{\xi \in \Xi_2} -\lambda d^p(\xi,\data_i) \},
\end{align*}
where $G$ is defined in \eqref{eq:dist_func}. Now, if $\data_i \in \Xi_2$, the second term is $0$. Alternatively, if $\data_i \in \Xi_1$, then $G(\data_i) = 0$ and the second term is nonpositive, in which case, the maximum evaluates to $1$. Accordingly, we have $s_i = \max\{1-\lambda G(\data_i), 0\}$. Thus, $\widehat{X}_{\mathtt{DCP}}$ \eqref{eq:def_drccp_dd} is equivalently given by \eqref{eq:drccp_dist_exact}. 

For the second reformulation, let $\widehat{X}'_{\mathtt{DCP}}$ denote the set given in \eqref{eq:drccp_cvar_exact}. We first show that $\widehat{X}_{\mathtt{DCP}} \subseteq \widehat{X}'_{\mathtt{DCP}}$. 
Let $x \in \widehat{X}_{\mathtt{DCP}}$ as stated in \eqref{eq:drccp_dist_exact}. Note that we must have $\lambda > 0$. Suppose otherwise, and let $\lambda = 0$. Then, $s_i = 1$ for $i \in [N]$, and consequently, we have $\alpha \geq 1$; a contradiction. Consequently, we can replace $\lambda$ in \eqref{eq:drccp_dist_exact} by $\frac{1}{t} > 0$, and obtain 
\begin{align}
& \frac{\theta^p}{t} + \frac{1}{N} \sum^N_{i=1} \max\left\{1-\frac{G(x,\data_i)}{t},0\right\} \leq \alpha \label{eq:int_exact_reform1}
\\ \iff & \frac{\theta^p}{\alpha} - t + \frac{1}{\alpha N} \sum^N_{i=1} \max\left\{-G(x,\data_i)+t,0\right\} \leq 0 \label{eq:int_exact_reform2}
\\ \implies & \frac{\theta^p}{\alpha} + \CVaR{\Pbhat_N}_{1-\alpha}(-G(x,\xi)) \leq 0 \nonumber
\end{align}
following the definition of conditional value-at-risk \eqref{eq:CVaR-def}; note that we can replace $t$ with $-t$ without loss of generality since the infimum in \eqref{eq:CVaR-def} is over $\Rb$. As a result, $x \in \widehat{X}'_{\mathtt{DCP}}$. 

It remains to show $\widehat{X}'_{\mathtt{DCP}} \subseteq \widehat{X}_{\mathtt{DCP}}$. Let $x \in \widehat{X}'_{\mathtt{DCP}}$. Then, 
$$ \frac{\theta^p}{\alpha} + \inf_{t \in \Rb} \left\{t + \frac{1}{\alpha N} \sum^N_{i=1} (-G(x,\data_i)-t)_{+} \right\}\leq 0. $$
From the fact that $\cl(\xi: F(x,\xi) > 0)$ is nonempty, we have $G(x,\data_i) < \infty$ for $i \in [N]$. As a result, we have $t + \frac{1}{\alpha N} \sum^N_{i=1} (-G(x,\data_i)-t)_{+} \to \infty$ as $|t| \to \infty$. Accordingly, there exists $\bar{t} \in \Rb$ such that
$$ \frac{\theta^p}{\alpha} + \bar{t} + \frac{1}{\alpha N} \sum^N_{i=1} (-G(x,\data_i)-\bar{t})_{+} \leq 0. $$
Since $G$ is nonnegative, we must have $\bar{t} < 0$. Consequently, we can define $\lambda = -\frac{1}{\bar{t}} > 0$, which implies $x \in \widehat{X}_{\mathtt{DCP}}$ as stated in \eqref{eq:drccp_dist_exact}. Therefore, $\widehat{X}_{\mathtt{DCP}} = \widehat{X}'_{\mathtt{DCP}}$.
\end{proof}

The condition~\eqref{eq:F-growth} is met if $F$ is bounded or $\xi \mapsto F(x,\xi)$ is Lipschitz for every $x \in X$ with $p=1$.
In~\cite{ahmed2018bicriteria}, authors show that DRCCPs under Wasserstein ambiguity sets \eqref{eq:def_drccp_dd} are strongly NP-Hard even for affine $F$. 
In light of this fact, we now focus on developing tractable approximations of DRCCPs using CVaR of the constraint function. 

\section{CVaR Approximation of DRCCPs}\label{sec:main-result}

When $F$ is convex in $x$, the CVaR approach of~\cite{nemirovski2006convex} provides a convex inner approximation of the feasibility set of the original (DR)CCP (see Section \ref{subsec:cc-approx} for details). In the remainder of the paper, we study this CVaR approximation of the DRCCP \eqref{eq:def_drccp_dd} under the following assumptions. 

\begin{assumption}\longthmtitle{$F$ is convex-bounded}\label{ass:main1}
The set $\Xi$ is a subset of $\real^m$. The function $F: \Rb^n \times \Xi \to \Rb$ satisfies:
\begin{enumerate}
\item for every $\xi \in \Xi$, $x \mapsto F(x,\xi)$ is convex on $X$,
\item for every $x \in X$, $\xi \mapsto F(x,\xi)$ is bounded on $\Xi$.
\end{enumerate}
\end{assumption}
Note that the second property in the above assumption implies~\eqref{eq:F-growth}. 

Following our earlier discussion in Section~\ref{subsec:cc-approx}, the CVaR approximation of the DRCCP \eqref{eq:def_drccp_dd} is 
\begin{equation}\label{eq:def_drccp_approx}
  \begin{aligned}
      & {\min}\{c^\intercal x: x\in \widehat{X}_{\mathtt{CDCP}}\}, \qquad \text{where} \\
      & \widehat{X}_{\mathtt{CDCP}}\! := \!\setdefb{x \!\in \!X \!}{ \!\underset{\Pb \in \MM^\theta_N}{\! \sup} \!\underset{t \in \Rb}{\inf} [\Eb_\Pb [(F(x,\xi)+t)_+]\!-\!t\alpha]
\!\leq\!0}.
  \end{aligned}
\end{equation} 
We start by reformulating the expression of $\Xcdcp$ and establishing its convexity. First we show that the $\inf$ and the $\sup$ in the constraint of \eqref{eq:def_drccp_approx} can be interchanged. The proof is an application of the min-max theorem due to \cite{shapiro2002minimax} stated as Theorem \ref{thm:min-max-shapiro} in Section \ref{subsec:dro}.

\begin{lemma}\longthmtitle{Min-max equality for the constraint
    function}\label{lemma:drccp_minmax}
Suppose Assumption~\ref{ass:main1} holds. Then for every $x \in X$, we have
  \begin{align}
    & \underset{\Pb \in \MM^\theta_N}{\sup} \,  \underset{t \in \Rb}{\inf} \,
    \Eb_{\Pb} [(F(x,\xi)+t)_+ -t\alpha] \nonumber
    \\ = & \underset{t \in \Rb}{\inf} \,
    \underset{\Pb \in \MM^\theta_N}{\sup} \, \Eb_{\Pb} [(F(x,\xi)+t)_+ -t\alpha]. \label{eq:min-max-equality}
  \end{align}
\end{lemma}
\begin{proof}
  We suppress the variable $x$ in the proof for better readability.
  We verify that the hypotheses of the min-max theorem (Theorem
  \ref{thm:min-max-shapiro}) hold.

  Drawing the parallelism in notation between our case and
  Theorem~\ref{thm:min-max-shapiro}, note that here $\Rb$ plays the
  role of both $T$ and $V$; $\MM^\theta_N$ that of $\MM$; and the
  function $g$ is $g(t,\xi):=(F(\xi)+t)_+ - t\alpha$. Following Proposition~\ref{proposition:wasserstein_compact}, $\MM^\theta_N$ is weakly compact.
  
  Note that $g$ is
  continuous as $F$ is so. Further since $F$ is bounded, for every
  $t \in \Rb$, the function $\xi \mapsto g(t,\xi)$ is bounded and
  $\sup_{\Pb \in \MM^\theta_N} \Eb_{\Pb} [g(t,\xi)] <
  \infty$. Finally, for every $\xi \in \Xi$, $t \mapsto g(t,\xi)$ is convex. Thus, to conclude the proof it remains to
  show that the infimum on the right-hand side
  of~\eqref{eq:min-max-equality} is attained. Define the function
  \begin{align*}
    h(t):= \underset{\Pb \in \MM^\theta_N}{\sup} \Eb_{\Pb} [(F(\xi)+t)_+ -t\alpha].
  \end{align*}
  Note that for any $\Pb \in \MM^\theta_N$, the function
  $t \mapsto \Eb_{\Pb} [(F(\xi)+t)_+ - t \alpha]$ is convex and
  real-valued. Since $h$ is supremum over a family of such functions,
  $h$ too is convex and real-valued. Hence, $h$ is continuous. Further
  note that $(F(\xi)+t)_+ - t \alpha \to \infty$ as
  $\abs{t} \to \infty$. This fact along with boundedness of $F$
  implies $h(t) \to \infty$ as $\abs{t} \to \infty$. Thus,
  $\inf_{t \in \Rb} h(t)$ exists, concluding the proof.
\end{proof}

Next, using the min-max equality established above and the strong
duality result of distributionally robust optimization presented in
Section~\ref{subsec:dro}, we obtain the following convex reformulation of the CVaR approximation of DRCCP~\eqref{eq:def_drccp_approx}.

\begin{proposition}\longthmtitle{Convex reformulation of~\eqref{eq:def_drccp_approx}}\label{pr:convex-DRCCP}
Under Assumption~\ref{ass:main1}, the CVaR approximation of the DRCCP problem \eqref{eq:def_drccp_approx} is equivalent to the following convex program 
\begin{equation}\label{eq:def_drccp_approx_reform}
\begin{aligned}
      \min & \quad \!\!\!\!c^\intercal x
      \\ \st & \quad \!\!\!\!\lambda \theta^p+\frac{1}{N} \sum^N_{i=1} s_i \leq t\alpha,
      \\ & \quad \!\!\!\!s_i \ge \underset{\xi \in \Xi}{\sup}[F(x,\xi)+t \!- \!\lambda d^p(\xi,\data_i)], \forall i \in [N],
      \\ & \quad \!\!\!\!\lambda \geq 0, t \in \Rb, x \in X, s_i \geq 0, \forall i \in [N].
\end{aligned}
\end{equation}  
Specifically, $x$ lies in the feasibility set of~\eqref{eq:def_drccp_approx} if and only if there exists $(\lm, t, \{s_i\}_{i=1}^N)$ such that $(x,\lm, t, \{s_i\}_{i=1}^N)$ is a feasible point for~\eqref{eq:def_drccp_approx_reform}.
\end{proposition}
\begin{proof}
We evaluate the constraint in \eqref{eq:def_drccp_approx} as
\begin{align}
& \underset{\Pb \in \MM^\theta_N}{\sup} \underset{t \in \Rb}{\inf}
[\Eb_\Pb[(F(x,\xi)+t)_+] -t\alpha] \nonumber
\\ = & \underset{t \in \Rb}{\inf} \underset{\Pb \in \MM^\theta_N}{\sup} [\Eb_\Pb[(F(x,\xi)+t)_+] -t\alpha] \nonumber
\\ = & \underset{t \in \Rb}{\inf} \underset{\lambda \geq 0}{\inf} [\lambda \theta^p - t\alpha \nonumber 
\\ & \textstyle \qquad +\frac{1}{N} \sum^N_{i=1} \underset{\xi \in \Xi}{\sup} [(F(x,\xi)+t)_+ - \lambda d^p(\xi,\data_i)]]. \label{eq:cvar_const_full}
\end{align}
The first equality follows as a consequence of Lemma \ref{lemma:drccp_minmax}. The second equality is a consequence of the strong duality result in Theorem \ref{theorem:gao_duality}; note that since $F$ is bounded, the condition~\eqref{eq:F-growth} holds (including when $\Xi$ is not bounded). Furthermore, the infimum over $\lambda \geq 0$ is attained following Theorem \ref{theorem:gao_duality}. Thus, $\widehat{X}_{\mathtt{CDCP}}$ is equivalent to the set
\begin{equation}\label{eq:def_drccp_approx_reform2}
\Pi_x\left\{\begin{array}{l}  x \in X, \\ \lambda \ge 0, \\ t \in \Rb \\ \{s_i\}_{i=1}^N \end{array} \Bigg| \begin{array}{l} \lambda \theta^p + \frac{1}{N} \sum_{i=1}^N s_i \le t\alpha, \\ s_i \ge \underset{\xi \in \Xi}{\sup} [(F(x,\xi)+t)_+ 
\\ \qquad \qquad -\lambda d^p(\xi,\data_i)], \forall i \in [N] \end{array} \right\},
\end{equation}
where $\Pi_x$ gives the $x$-component of the argument. 

Now observe that for a given $(x,\lambda,t)$ and $i \in [N]$, 
\begin{align*}
s_i & \ge \max\{\underset{\xi \in \Xi_1}{\sup} F(x,\xi)+t -\lambda d^p(\xi,\data_i),\underset{\xi \in \Xi_2}{\sup} -\lambda d^p(\xi,\data_i)\},
\end{align*}
where $\Xi_1 = \{\xi \in \Xi: F(x,\xi)+t \geq 0\}$, and $\Xi_2 = \Xi \setminus \Xi_1$. We distinguish between the following two cases.

Suppose $\data_i \in \Xi_1$. Then, $\underset{\xi \in \Xi_2}{\sup} -\lambda d^p(\xi,\data_i) < 0$ and $\underset{\xi \in \Xi_1}{\sup} F(x,\xi)+t -\lambda d^p(\xi,\data_i) =  \underset{\xi \in \Xi}{\sup} F(x,\xi)+t -\lambda d^p(\xi,\data_i) > 0$. On the other hand, if $\data_i \in \Xi_2$, we have $\underset{\xi \in \Xi_2}{\sup} -\lambda d^p(\xi,\data_i) = 0 > \underset{\xi \in \Xi_2}{\sup} F(x,\xi)+t -\lambda d^p(\xi,\data_i)$. In both cases, we have
$$ s_i \ge \max\{ \underset{\xi \in \Xi}{\sup} [F(x,\xi)+t -\lambda d^p(\xi,\data_i)],0\}. $$
This concludes the proof.
\end{proof}

The above result shows that the CVaR approximation of DRCCPs under Wasserstein ambiguity sets can be reformulated as a convex optimization problem. However, the constraints involving $s_i$ in \eqref{eq:def_drccp_approx_reform} involve supremum operators. In the remainder of the paper, we develop tractable reformulations and algorithms to solve \eqref{eq:def_drccp_approx_reform} under suitable assumptions on the constraint function $F$.

\section{Reformulations and Algorithms for Several Classes of Constraint Functions}

\subsection{$F$ Piecewise Affine in Uncertainty}\label{sec:cvar-affine}

We now present a tractable reformulation \eqref{eq:def_drccp_approx_reform} when $F$ is the maximum of a set of functions that are affine in $\xi$. The analysis is inspired by a similar reformulation in \cite{peyman2017wasserstein} shown for distributionally robust stochastic optimization.

\begin{proposition}\longthmtitle{Reformulation of DRCCP for piecewise affine
    $F$}\label{prop:affine-reform}
  Let $\Xi = \setdef{\xi \in \Rb^m}{C \xi \leq h}$ be compact, where 
  $C \in \Rb^{p \times m}$ and $h \in \Rb^p$ for some $p > 0$. Suppose that for some positive integer $K$, 
  $F(x,\xi) := \max_{k \le K} x^\intercal A_k \xi + b_k(x)$, where
  $A_k \in \real^{n \times m}$ and $b_k:\Rb^n \to \Rb$ are convex functions
  for all $k \in [K]$. Let the ambiguity set $\MM^\theta_N$ be
  defined using the $1$-Wasserstein metric and $d$ be the standard
  Euclidean distance. Then, the DRCCP~\eqref{eq:def_drccp_approx_reform} is
  equivalent to the following tractable convex optimization problem
  \begin{equation*}\label{eq:def_drccp_approx_reform_affine}
    \begin{aligned}
     & \min 
      \quad c^\intercal x
      \\
      & \st \lambda \theta + \frac{1}{N}
      \sum^N_{i=1} s_i \le t \alpha, 
      \\
      & \qquad \bigl(b_k(x) + t + (x^\intercal A_k - C^\intercal \eta_{ik})^\intercal \data_i + \eta_{ik}^\intercal h\bigr)_+ \le s_i,
      \\
      & \qquad \norm{x^\intercal A_k - C^\intercal \eta_{ik}} \le \lambda, \eta_{ik} \ge 0,
      \\
      & \qquad x\in X, t\in\Rb, \lambda \geq 0,
    \end{aligned}
    \end{equation*}
    where the inequality involving the set of variables $\eta_{ik}$ hold for $i \in \until{N}$ and $k \in \until{K}$.
\end{proposition}
\begin{proof}
Note that the hypotheses here imply Assumption~\ref{ass:main1} is met. Then following Proposition \ref{pr:convex-DRCCP} and \eqref{eq:def_drccp_approx_reform}, we focus on reformulating the constraints involving the auxiliary variables $s_i, i \in [N]$. In particular, for piecewise maximum of affine functions, we have 
\begin{align}
s_i & \ge (\underset{\xi \in \Xi}\sup [\max_{k\in\until{K}}
      \{ x^\top A_k  \xi + b_k(x)\}+t  -\lambda \norm{\xi -
        \data_i}])_+ \nonumber
\\   & = \max_{k\in\until{K}}\bigl( b_k(x) + t + \sup_{\xi \in \Xi} [x^\intercal A_k \xi - \lambda \norm{\xi - \data_i}]\bigr)_+, \nonumber
\\ & \geq \bigl( b_k(x) + t + \sup_{\xi \in \Xi} [x^\intercal A_k \xi - \lambda \norm{\xi - \data_i}]\bigr)_+, \label{eq:simplification-aff-1}
\end{align}
for every $k \in [K]$. The second equality interchanges the $\sup$ and the $\max$. We now compute
  \begin{align}
    & \sup_{\xi \in \Xi} [x^\intercal A_k \xi - \lambda \norm{\xi - \data_i}] \notag
    \\ \overset{(a)}{=} & \sup_{\xi \in \Xi}[x^\intercal A_k \xi-
      \sup_{\norm{z_{ik}} \le \lambda} z_{ik}^\intercal (\xi - \data_i)] \notag
    \\
     \overset{(b)}{=} & \inf_{\norm{z_{ik}} \le \lambda} \bigl[ z_{ik}^\intercal \data_i
      + \sup_{\xi \in \Xi} [(x^\intercal A_k- z_{ik})^\intercal \xi] \bigr] \notag
    \\
     \overset{(c)}{=} & \inf_{\norm{z_{ik}} \le \lambda} \bigl[ z_{ik}^\intercal \data_i
      + \inf_{\eta_{ik} \ge 0, z_{ik} = x^\intercal A_k - C^\intercal \eta_{ik}} \eta_{ik}^\intercal h \bigr] \notag
    \\
    = & \inf_{\substack{\norm{z_{ik}} \le \lambda, \eta_{ik} \ge 0, \\ z_{ik} = x^\intercal A_k - C^\intercal \eta_{ik}}}
      [ z_{ik}^\intercal \data_i + \eta_{ik}^\intercal h] \notag
    \\
    = & \inf_{\substack{\eta_{ik} \ge 0 \\ \norm{x^\intercal A_k - C^\intercal \eta_{ik}} \le \lambda}}
      \bigl[ (x^\intercal A_k - C^\intercal \eta_{ik})^\intercal \data_i + \eta_{ik}^\intercal h \bigr].
      \label{eq:simplification-aff-2}
  \end{align}
  Here, (a) uses the definition of the norm, (b) follows by $\inf$-$\sup$ interchange due to~\cite[Corollary 37.3.2]{rockafellar1970convex-analysis}, and (c) writes the dual form of the
  inner linear program (from 
  $\Xi = \setdef{\xi \in \Rb^m}{C\xi \le
    h}$). Substituting~\eqref{eq:simplification-aff-2}
  in~\eqref{eq:simplification-aff-1}, we obtain
  \begin{align}
     s_i \geq & \displaystyle \Bigl( b_k(x) + t +\inf_{\substack{\eta_{ik} \ge 0 \\ \norm{x^\intercal A_k- C^\intercal \eta_{ik}} \le \lambda}}
      [(x^\intercal A_k - C^\intercal \eta_{ik})^\intercal
      \data_i \notag
      \\ 
      & \qquad \qquad \qquad \quad + \eta_{ik}^\intercal h] \Bigr)_+, \forall k \in \until{K}. \label{eq:s-ineq}
  \end{align}
The above equation holds if and only if there exists $\eta_{ik} \ge 0$ for all $k \in \until{K}$ such that for all $k \in \until{K}$, 
\begin{equation}\label{eq:s-ineq-2}
\begin{aligned}
& s_i \geq \Bigl( b_k(x) + t + (x^\intercal A_k - C^\intercal \eta_{ik})^\intercal
\data_i + \eta_{ik}^\intercal h) \Bigr)_+, 
\\
&   \norm{x^\intercal A_k - C^\intercal \eta_{ik}} \le \lambda, 
\end{aligned}
\end{equation}
The ``if" part in the above statement is straightforward. For the ``only if" part consider two cases for any $k \in \until{K}$: either the $\inf$ in~\eqref{eq:s-ineq} is attained or it is not. In the former, the optimizer of the $\inf$ satisfies~\eqref{eq:s-ineq-2}. In the later the optimal value of $\inf$ is $-\infty$ in which case the constraint~\eqref{eq:s-ineq} reads as $s_i \ge 0$. Thus, one can find $\eta_{i,k}$ such that the expression inside $( \cdot )_+$ is negative in~\eqref{eq:s-ineq-2} and so the constraint in~\eqref{eq:s-ineq-2} reduces to $s_i \ge 0$. This concludes the proof.
\end{proof}

\begin{remark}\longthmtitle{Comparison with literature and exactness of CVaR approximation}
	{\rm 
In \cite{xie2018drccp,kuhn2018drccp}, authors derive the reformulation given in Proposition~\ref{prop:affine-reform} for the case when $\Xi = \real^m$. In addition, they show that when $\Xi = \Rb^m$ and $N\alpha \leq  1$, the CVaR approximation is exact, i.e., $\widehat{X}_{\mathtt{DCP}} = \widehat{X}_{\mathtt{CDCP}}$. \oprocend}
\end{remark}
In the following subsection, we present an algorithm that solves CVaR approximation of DRCCPs when the constraint function is concave in uncertainty.

\subsection{$F$ Concave in Uncertainty}\label{sec:concave-uncertainty}

Here we aim to develop an algorithm for~\eqref{eq:def_drccp_approx_reform} when $F$ is concave in $\xi$. The roadblock in solving~\eqref{eq:def_drccp_approx_reform} is the supremum operator present in the constraint that makes implementing first- or second-order methods almost impossible. To construct the algorithm, we view~\eqref{eq:def_drccp_approx_reform} as a semi-infinite program and employ the central cutting surface algorithm proposed in~\cite{mehrotra2014cutting}. The algorithm requires the feasibility set of the problem to be compact. Thus, as a first step, we identify a compact set which contains the optimizers of~\eqref{eq:def_drccp_approx_reform}. Our results hold under the following assumption.

\begin{assumption}\longthmtitle{$F$ concave in uncertainty and existence of robustly feasible point}\label{ass:compact-robust-feasibility}
	The sets $X$ and $\Xi$ are compact. For every $x \in X$, the function $\xi \mapsto F(x,\xi)$ is concave. There exists $\bar x \in X$ such that $F(\bar x,\xi) \le - \delta < 0$ for all $\xi \in \Xi$. 
\end{assumption}

The next result provides bounds on the optimizers of~\eqref{eq:def_drccp_approx_reform}. 
\begin{lemma}\longthmtitle{Optimizers of~\eqref{eq:def_drccp_approx_reform} belong to a compact set}\label{le:opt-bound}
Under Assumption~\ref{ass:main1} and~\ref{ass:compact-robust-feasibility}, the optimizers of~\eqref{eq:def_drccp_approx_reform} belong to the set
	$X \times [0,t^M] \times [0,\lm^M] \times [0,\alpha Nt^M]^N$,
	where 
	\begin{align*}
	t^M & := \frac{1}{1-\alpha} \sup_{x \in X, \xi \in \Xi} - F(x,\xi), \quad \text{and} \quad \lm^M = \frac{\alpha t^M}{\theta^p}.
	\end{align*}
\end{lemma}
\begin{proof}
	Let $(\xo,\topt,\lmo,\{s^\star_i\}^N_{i=1})$ be an optimizer of~\eqref{eq:def_drccp_approx_reform}. By definition, $\xo \in X$. For the sake of contradiction, assume $\topt \not \in [0,t^M]$. Note that for each $i \in \until{N}$,
	\begin{align*}
	\sup_{\xi \in \Xi}  (F(x,\xi) + t)_+ \! - \!  \lm d^p(\xi,\data_i) & \! \ge \! ( F(x,\data_i) + t)_+ \! - \! \lm d^p(\data_i,\data_i) 
	\\
	& = ( F(x,\data_i) + t)_+ \ge 0.
	\end{align*}
	Therefore, the left-hand side of the first constraint in~\eqref{eq:def_drccp_approx_reform} is lower bounded by $\lm \theta^p - t \alpha$. If $\topt < 0$, then the constraint is violated as $\lm \ge 0$. The other possibility is $\topt > t^M$. Since $\alpha < 1$, we have 
	$t^M > \sup_{x \in X, \xi \in \Xi} - F(x,\xi)$
	which implies 
	$\topt > -F(x,\xi), \quad \forall x \in X, \xi \in \Xi$.
	Using this fact, we get 
	\begin{align}
	(F(x,\xi) + \topt)_+ = F(x,\xi) + \topt, \quad \forall x \in X, \xi \in \Xi. \label{eq:kill-proj}
	\end{align}
	To arrive at the contradiction, we will show that the constraint in~\eqref{eq:def_drccp_approx_reform} is violated for such a choice of $\topt$. Note that
	\begin{align}
	\lmo \theta^p - \topt \alpha &+ \frac{1}{N} \sum_{i=1}^N \sup_{\xi \in \Xi} (F(\xo,\xi)+\topt)_+ - \lmo d^p(\xi,\data_i) \notag
	\\
	& \overset{(a)}{\ge} \lmo \theta^p - \topt \alpha + \frac{1}{N} \sum_{i=1}^N (F(\xo,\data_i) + \topt) \notag
	\\
	& \overset{(b)}{\ge} \topt (1-\alpha) + \inf_{x \in X, \xi \in \Xi} F(x,\xi) \notag
	\\
	& = \topt(1-\alpha) - \sup_{x \in X, \xi \in \Xi} - F(x,\xi) > 0, \label{eq:constraint-violation-t}
	\end{align}
	where in (a) we lower bound the supremum in each $i$-th term by substituting $\xi$ with $\data_i$ and then use~\eqref{eq:kill-proj}. In (b), we use nonnegativity of $\lmo$ and a lower bound on $F$. From~\eqref{eq:constraint-violation-t}, we conclude that $\topt \in [0,t^M]$. To show that $\lmo \in [0,\lm^M]$, recall that the left-hand side of the first constraint of \eqref{eq:def_drccp_approx_reform} is lower bounded by $\lm \theta^p - t \alpha$. For the constraint to be feasible we would require $\lm \theta^p - t \alpha \le 0$ implying $\lm \le t \alpha /\theta^p$. The bound on $\lmo$ then follows by using the bound on $\topt$. 	
	
Finally, since $\lambda^\star \geq 0$ and $s^\star_i \geq 0, \forall i \in [N]$, the first constraint of \eqref{eq:def_drccp_approx_reform} implies $s^\star_i \leq \alpha N t^M, \forall i \in [N]$.
\end{proof}

Using the above result, one can restrict the feasibility set of~\eqref{eq:def_drccp_approx_reform} without disturbing its optimizers. We denote the decision variables of \eqref{eq:def_drccp_approx_reform} as $y:=(x,t,\lm,\setr{s_i}_{i=1}^N)$, and its feasibility set as the compact set $Y := X \times [0,t^M] \times [0,\lm^M] \times [0,\alpha Nt^M]^N$. The optimization problem~\eqref{eq:def_drccp_approx_reform} over the restricted domain written as semi-infinite program is
\begin{equation}\label{eq:semi-inf}
\begin{aligned}
\min & \quad \!\!\!\! c^\intercal x
\\
\st & \quad \!\!\!\! \lambda \theta^p + \frac{1}{N} \sum_{i=1}^N s_i  \le t \alpha,
\\ 
& \quad \!\!\!\! s_i \ge F(x,\xi)+t \!- \!\lambda d^p(\xi,\data_i), \forall \xi \in \Xi, \, \forall i \in [N],
\\
& \quad \!\!\!\! (x,t,\lambda,\{s_i\}^N_{i=1}) \in Y.
\end{aligned}
\end{equation}
Now, for each $i \in \until{N}$, we define the function 
\begin{align*}
H_i(y,\xi):= F(x,\xi)+t - \lm d^p(\xi,\data_i) - s_i.
\end{align*}
Next, set the parameter $B > 0$ satisfying
\begin{align*}
	B > \norm{g^i(y,\xi)}, \forall y \in Y, \, \forall \xi \in \Xi, \forall i \in \until{N}
\end{align*}
where $g^i(y,\xi) = (g^i_y(y,\xi) , g^i_\xi (y,\xi)) \in \partial_y H_i(y,\xi) \times \partial_\xi H_i(y,\xi)$. That is, $B$ bounds the set of subgradients of $H_i$, for all $i$, over the feasibility set $Y$. Semi-infinite optimization problems are difficult to solve in general. Thus, our objective is to design an algorithm that can find an approximate solution to the problem~\eqref{eq:semi-inf}. This is made precise below.

\begin{definition}\longthmtitle{Approximate feasibility and optimality of \eqref{eq:semi-inf}}\label{def:semi_inf_opt}
We say that a point $y = (x,t,\lm,\setr{s_i}_{i=1}^N) \in Y$ is $\eta$-feasible for the problem~\eqref{eq:semi-inf} if it satisfies 
\begin{gather*}
\lambda \theta^p + \frac{1}{N} \sum_{i=1}^N s_i  \le t \alpha,
\\
s_i + \eta \ge F(x,\xi)+t \!- \!\lambda d^p(\xi,\data_i), \forall \xi \in \Xi, \, \forall i \in [N],
\end{gather*}
Further, a point $(x_\eta^\star, t_\eta^\star, \lm_\eta^\star, \setr{s_{i,\eta}^\star}_{i=1}^N)$ is an $\eta$-optimal solution of~\eqref{eq:semi-inf} if it is $\eta$-feasible and $c^\intercal x_\eta^\star \le c^\intercal x^\star$ where $(x^\star, t^\star, \lm^\star, \setr{s_i^\star}_{i=1}^N)$ is an optimizer of~\eqref{eq:semi-inf}.
\end{definition}

We propose an algorithm that finds an $\eta$-optimal solution of~\eqref{eq:semi-inf}. Our scheme involves solving a convex optimization problem, termed the master problem, at every iteration of the algorithm. The master problem for the $k$th iteration is 
\begin{equation}\label{eq:master}
\begin{aligned}
\max & \quad \sigma 
\\
\st & \quad c^\intercal x + \sigma \le M^{(k-1)},
\\
& \quad \lambda \theta^p + \frac{1}{N} \sum_{i=1}^N s_i \le t \alpha,
\\
& \quad H_i(y,\xi_i) + \sigma B \leq 0, \forall \xi_i \in Q^{(k-1)}_i,
\\
& \quad (x,t,\lambda,\{s_i\}^N_{i=1}) \in Y.
\end{aligned}
\end{equation}
Various terms of the above optimization are introduced below where we elaborate on the steps of Algorithm~\ref{alg:cut-surface}. 

Each iteration $k$ starts by solving~\eqref{eq:master}. The aim of this step is to find $y^{(k)}$ that is {\it robustly} feasible to the constraints sampled till the $k$th iteration, $Q_i^{(k-1)}$, $i \in \until{N}$, and that also improves the upper bound on the objective value $M^{(k-1)}$. The variable $\sigma^{(k)}$ denotes this improvement. Upon solving~\eqref{eq:master}, two cases arise. First, $y^{(k)}$ is $\eta$-feasible and so, there does not exist, for any $i$, a violating constraint $\xi_i^{(k)}$ that can be added to $Q_i^{(k-1)}$. In this case, we move to Step~\ref{step:opt-cut} where the constraint set is kept same, the best estimate of the optimizer $\tilde{y}^{(k-1)}$ is updated to the $\eta$-feasible solution found in this iteration, and the upper bound is updated. In the second case, a violating constraint is determined for each $i$ (if possible) in Step~\ref{step:find-xi}. Subsequently, in Step~\ref{step:feas-cut}, the constraint set is updated while the best estimate of the optimizer and the upper bound are kept the same. The algorithm converges when the objective value cannot be improved anymore over the set of all $\eta$-feasible solutions.

\begin{algorithm}
	\SetAlgoLined
	\DontPrintSemicolon
	\SetKwInOut{giv}{Data} \SetKwInOut{ini}{Input}
	\SetKwInOut{state}{State} \SetKwInput{start}{Initialize}
	\ini{Assumption~\ref{ass:compact-robust-feasibility} holds. For a given $y$ and $i \in \until{N}$, whenever $\sup_{\xi \in \Xi} H_i(y,\xi) > \eta$, then there exists an oracle that determines a point $\xi \in \Xi$ such that $H_i(y,\xi) > 0$.}
	\start{Set $k =1$, $M^{(0)} = U:= \max_{x \in X} c^\intercal x$, $Q^{(0)}_i = \emptyset$ for all $i \in \until{N}$, $\yt^{(0)} = 0$.}
	Determine the optimizer $(y^{(k)}, \sigma^{(k)})$ of the master problem~\eqref{eq:master} \; \label{step:solve-master}
	If $\sigma^{(k)} = 0$, stop and return $\tilde{y}^{(k-1)}$ \; 
	For each $i \in \until{N}$, find (if possible) $\xi_i^{(k)} \in \Xi$ such that $H_i(y^{(k)}, \xi_i^{(k)}) > 0$ and then go to Step~\ref{step:feas-cut}; if no such point exists for any $i$, then go to Step~\ref{step:opt-cut} \;  \label{step:find-xi}
	Set for each $i \in \until{N}$, $Q_i^{(k)} = Q_i^{(k-1)} \cup \{\xi_i^{(k)}\}$ whenever a point $\xi_i^{(k)}$ is found in Step~\ref{step:find-xi}, otherwise $Q_i^{(k)} = Q_i^{(k-1)}$; Set $\tilde{y}^{(k)} = \tilde{y}^{(k-1)}$ and $M^{(k)} = M^{(k-1)}$; Go to Step~\ref{step:last} \; \label{step:feas-cut}
	Set $Q^{(k)} = Q^{(k-1)}$, $\tilde{y}^{(k)} = y^{(k)}$, and $M^{(k)} = c^\top x^{(k)}$ \; \label{step:opt-cut}
	Increase $k$ by one and go to Step~\ref{step:solve-master} \; \label{step:last}
	\caption{A central cutting-surface algorithm for~\eqref{eq:semi-inf}}
	\label{alg:cut-surface}
\end{algorithm}


The next result states the correctness of Algorithm~\ref{alg:cut-surface}. The proof involves arguments similar in reasoning to those presented in~\cite{mehrotra2014cutting}. An important ingredient is the compactness of the feasibility set which we achieved due to Lemma~\ref{le:opt-bound}. 
\begin{proposition}\longthmtitle{Convergence guarantee of Algorithm~\ref{alg:cut-surface}}\label{pr:convergence}
	Let Assumptions~\ref{ass:main1} and~\ref{ass:compact-robust-feasibility} hold. Consider the iterates $(\yt^{(k)})_{k=1}^\infty$ generated by Algorithm~\ref{alg:cut-surface}. 
	\begin{enumerate}
		\item If Algorithm~\ref{alg:cut-surface} terminates in the $k$th iteration, then $\yt^{(k-1)}$ is an $\eta$-optimal solution to~\eqref{eq:semi-inf}.
		\item If Algorithm~\ref{alg:cut-surface} does not terminate, then there exists an index $\hat{k}$ such that the sequence $\setr{\yt^{(\hat{k} +i)}}_{i=1}^\infty$ consists entirely of $\eta$-feasible solutions of~\eqref{eq:semi-inf}.
		\item If Algorithm~\ref{alg:cut-surface} does not terminate, then the sequence $\setr{\yt^{(k)}}_{k=1}^\infty$ has an accumulation point, and each accumulation point is an $\eta$-optimal solution to~\eqref{eq:semi-inf}. 
	\end{enumerate}
\end{proposition}
\subsection{$F$ Convex in Uncertainty}
\label{sec:convex}

We now consider $F$ to be convex in $\xi$. For this class of functions, unlike the case dealt in the previous section, the supremum present in the definition of the constraint set of \eqref{eq:def_drccp_approx_reform} is nonconvex, as it involves maximizing a difference of convex functions. In this section, we provide a convex inner approximation of~\eqref{eq:def_drccp_approx_reform} which is computable using standard convex optimization tools. We then compare the feasibility set of this convex inner approximation with two other feasibility sets obtained from sample based approaches for CCPs. We consider $\Xi \subseteq \Rb^m$ and the 1-Wasserstein distance in this section, i.e., $p=1$. The results rely on the following assumption.

\begin{assumption}\longthmtitle{Lipschitz in uncertainty}\label{assumption:f_convex}
	For every $x \in X$, the function $\xi \mapsto F(x,\xi)$ is convex. Moreover, there exists a convex function $\map{L_F}{X}{\Rb_{>0}}$, such that $\xi \mapsto F(x,\xi)$ is Lipschitz continuous with constant $L_F(x)$. 
\end{assumption}

Under the above assumption, we derive the following inner approximation of the feasibility set of the CVaR approximation of DRCCP $\Xcdcp$ given by \eqref{eq:def_drccp_approx}. 

\begin{lemma}\longthmtitle{Inner approximation of $\Xcdcp$}\label{lemma:hat_x_dcp_inclusion}
Let Assumptions \ref{ass:main1} and \ref{assumption:f_convex} hold. Define
\begin{equation*}
\Xcdcpin \!:= \!\setdefb{\!x \in \!X \!\!}{\!\theta\!L_F(x)\!+\!\underset{t \in \Rb}{\inf}\!\frac{1}{N}\!\!\sum_{\!i=1}^{\!N} (F(x,\widehat\xi_i)\!+\!t)_{\!+} \!\!- t\alpha\!\leq\!0\!}.
\end{equation*}
Then, $\Xcdcpin \subseteq \Xcdcp$ and these sets are equal when $\Xi = \Rb^m$.
\end{lemma}
\begin{proof}
	Suppose $\bar{x} \in \Xcdcpin$. Recall from the proof of Lemma \ref{lemma:drccp_minmax} that for each $i \in \until{N}$, $(F(\bar{x},\data_i)+t)_+ -t\alpha \to \infty$ as $|t| \to \infty$. Therefore, fixing $\bar{x}$, the infimum present in the inequality defining $\Xcdcpin$ is attained. That is, there exists $\bar{t} \in \real$ satisfying
	\begin{equation}\label{eq:tbar-convex-inequality}
	\theta L_F(\bar{x})+ \frac{1}{N} \sum_{i=1}^{N} (F(\bar{x},\data_i)+\bar{t})_+ -\bar{t}\alpha \leq 0. 
	\end{equation}
	Further, $\bar{t}$ should be positive as otherwise the above inequality will not hold. Note that under Assumption \ref{assumption:f_convex}, $\xi\mapsto (F(\bar{x},\xi)+\bar{t})_+ -\bar{t}\alpha$ is convex and Lipschitz continuous with constant $L_F(\bar{x})$. Therefore, we get 
	\begin{align}
	& \underset{t \in \Rb}{\inf} \underset{\Pb \in \MM^\theta_N}{\sup} [\Eb_\Pb[(F(x,\xi)+t)_+] -t\alpha] \notag
	\\
	&\qquad \qquad  \le \underset{\Pb \in \MM^\theta_N}{\sup} [\Eb_\Pb[(F(\bar{x},\xi)+\bar{t})_+] -\bar{t}\alpha] \notag
	\\ 
	& \qquad \qquad \leq \theta L_F(\bar{x})+ \frac{1}{N} \sum_{i=1}^{N} (F(\bar{x},\data_i)+\bar{t})_+ -\bar{t}\alpha, \label{eq:convex-inner-bound}
	\end{align}
	where the last inequality is due to Theorem 6.3 and Proposition 6.5 of~\cite{peyman2017wasserstein}. From~\eqref{eq:tbar-convex-inequality} and~\eqref{eq:convex-inner-bound} we conclude that $\bar{x} \in \Xcdcp$. The equality is due to the fact that when $\Xi = \Rb^m$, the inequality~\eqref{eq:convex-inner-bound} becomes an equality. 
\end{proof}

Observe that above, we upper bound the supremum over the Wasserstein ambiguity set in \eqref{eq:def_drccp_approx} with the sample average and a regularizer term. The proof is a consequence of \cite[Theorem 6.3,Proposition 6.5]{peyman2017wasserstein}. The Lipschitz continuity of $\xi \mapsto F(x,\xi)$ is a sufficient condition for \cite[Theorem 6.3]{peyman2017wasserstein}, and thus, Lemma \ref{lemma:hat_x_dcp_inclusion} may indeed hold for a more general class of functions.

Due to Lemma~\ref{lemma:hat_x_dcp_inclusion}, instead of minimizing the objective over $\Xcdcp$, one could perform the minimization over $\Xcdcpin$. The later problem is easier to deal with and the obtained solution will be feasible with respect to $\Xcdcp$ and hence $\Xdcp$. Consequently, the optimal value will provide an upper bound on the cost of \eqref{eq:def_drccp_approx_reform}.
We now compare the set $\Xcdcpin$ with the feasibility sets of the sample approximation approach \cite{luedtke2008sample}, and the scenario approach \cite{calafiore2005uncertain}. Given $\delta \in [0,1]$ and samples $\setr{\data_i}_{i=1}^N$, the sample approximation feasibility set is
\begin{align}
\Xsa{\delta} & := \setdefb{x \in X}{\frac{1}{N} \sum^N_{i=1} \mathbb{1}_{\{F(x,\widehat{\xi}_i ) \leq 0\}} \geq 1-\delta}.\label{eq:sample_feasible}
\end{align}
Specifically, if $x \in \Xsa{\delta}$, then at most $\delta$ fraction of samples $\setr{\data_i}$ violate the constraint $F(x,\xi) \leq 0$.
Similarly, given $\delta \ge 0$ and samples $\setr{\data_i}_{i=1}^N$, we define
\begin{equation}
\Xscp{\delta} := \setdefb{x \in X}{F(x,\widehat{\xi}_i) + \delta \leq 0, i \in {N}}. \label{eq:scenario_robust_feasible}
\end{equation}
Note that the feasibility set of the scenario program is $\Xscp{0}$. Thus, $\Xscp{\delta}$ defines a ``robust" scenario program, and for any $\delta > 0$, $\Xscp{\delta} \subseteq \Xscp{0}$. Also note that $\Xscp{0} = \Xsa{0}$. The main result of this subsection is stated below.

\begin{proposition}\longthmtitle{Comparison with $\Xcdcpin$}\label{prop:scp_robust_dcp_feasibility}
	Let Assumptions \ref{ass:main1} and \ref{assumption:f_convex} hold. Assume $L_F$ is constant over $X$. Let $t^* := \underset{x \in X, \xi \in \Xi}{\sup} -F(x,\xi)$, $\delta_1 :=  \alpha - \frac{\theta L_F}{t^*}$, and $\delta_2 := \frac{\theta L_F}{\alpha}$. Then, $\Xscp{\delta_2} \subseteq \Xcdcpin \subseteq \Xsa{\delta_1}$.
\end{proposition}
\begin{proof}
We first prove $\Xcdcpin \subseteq \Xsa{\delta_1}$. Let $\bar{x} \in \Xcdcpin$ and $J^N := \{i \in [N]|F(\bar{x},\widehat{\xi}_i) > 0\}$, i.e., $J^N$ is the set of indices of samples that violate the constraint $F(\bar{x},\xi) \leq 0$. By the definition of $\Xcdcpin$,
  \begin{align}\label{eq:tstar-ineq}
  L_F\theta+ \underset{t \in \Rb}{\inf} \frac{1}{N} \sum_{i=1}^{N} (F(\bar{x},\data_i)+t)_+ -t\alpha \leq 0
  \end{align} 
  Our first step is to show that $\bar{t}$, the point at which infimum is attained in the above expression, is at most $t^*$. 
  Note that for each $i \in \until{N}$, the function $t \mapsto (F(\bar{x},\data_i)+t)_+ - t \alpha$ is convex, has a unique minimizer at $-F(\bar{x},\data_i)$, and is strictly increasing in the region $t \ge -F(\bar{x},\data_i)$. Therefore, the function $t \mapsto \sum_{i=1}^N (F(\bar{x},\data_i)+t)_+ - t \alpha$ is strictly increasing in the region $t \ge \max_{i \in \until{N}} - F(\bar{x},\data_i)$, which contains $t \ge t^*$. Thus, $\bar{t} \le t^*$. Substituting $\bar{t}$ in~\eqref{eq:tstar-ineq} and removing the infimum gives \begin{align*}
  L_F\theta+ \frac{1}{N} \sum_{i=1}^{N} (F(\bar{x},\widehat\xi_i)+\bar{t})_+ -\bar{t}\alpha \leq 0.
  \end{align*}
  Rearranging the terms and using the definition of $J^N$ yields
  \begin{align*}
  & \bar{t}\alpha - L_F\theta \geq \frac{1}{N} \sum_{i \in J^N} (F(\bar{x},\widehat{\xi}_i) + \bar{t}) > \frac{|J^N| \bar{t}}{N}
  \\ 
  & \implies  \frac{|J^N|}{N} < \alpha - \frac{L_F \theta}{\bar{t}} \leq \alpha - \frac{L_F \theta}{t^*} 
  \\ 
  & \implies  {\frac{1}{N} \sum^N_{i=1} \mathbb{1}_{\{F(\bar{x},\widehat{\xi}_i ) > 0\}} < \alpha - \frac{L_F \theta}{t^*}}.
  \end{align*}
  The first implication uses the bound $\bar{t} \le t^*$ and the second uses the definition of $J^N$. This concludes the proof.

Now, let $\bar{x} \in \Xscp{\delta_2}$. By definition, $F(\bar{x},\widehat{\xi}_i) + \delta_2 \leq 0$ for all $i \in [N]$. Using this fact, we get  
	\begin{align*}
		&L_F\theta+ \underset{t \in \Rb}{\inf} \frac{1}{N} \sum_{i=1}^{N} (F(x,\widehat\xi_i)+t)_+ -t\alpha
		\\
		& \qquad \le  L_F\theta+ \frac{1}{N} \sum_{i=1}^{N} \left(F(\bar{x},\widehat\xi_i)+\frac{\theta L_F}{\alpha}\right)_+ - L_F \theta = 0. 
	\end{align*}
	The inequality holds as we have picked $t = \delta_2$ and removed the infimum operator. 
	Thus, we conclude that $\bar{x} \in \Xcdcpin$. 
\end{proof}

The above result shows that the feasibility set of the robust scenario program \eqref{eq:scenario_robust_feasible} is contained in the set $\Xcdcpin$. Furthermore, by the definition of the sample approximation set~\eqref{eq:sample_feasible}, the above result implies that if $\bar{x} \in \Xcdcpin$, then at most $\delta_1 < \alpha$ fraction of samples violate the constraint $F(x,\xi) \leq 0$. Both $\delta_1$ and $\delta_2$ depend on the Lipschitz constant, the probability of constraint violation $\alpha$, the Wasserstein radius, and $\delta_1$ depends additionally on $t^*$.

Independent of our work, \cite{xie2018drccp} showed the above relationships between the feasibility sets $\widehat{X}^{\mathtt{in}}_{\mathtt{CDCP}}$, $\widehat{X}_{\mathtt{SA},\alpha}$ and $\widehat{X}_{\mathtt{SCP},\delta}$ when the constraint function is affine in $x$ and $\xi$. We show that the above comparison holds more generally when the constraint function is convex in both $x$ and $\xi$.

We now present the following comparison between different feasibility sets studied in this paper. For $\delta_1 = \alpha - \frac{\theta L_F}{t^*}$ and $\delta_2 =  \frac{\theta L_F}{\alpha}$, we have

{\small 
	\begin{empheq}[box=\fbox]{equation}
\def\arraystretch{1}
\begin{array}{c c c c c}
\Xscp{0} & & = & & \Xsa{0} \\
\rotatebox[origin=c]{90}{$\subseteq$} & & & & \rotatebox[origin=c]{270}{$\subseteq$}\\
\widehat{X}_{\mathtt{SCP},\delta_2} & \subseteq & \widehat{X}^{\mathtt{in}}_{\mathtt{CDCP}} & \subseteq & \widehat{X}_{\mathtt{SA},\delta_1}\\
 & & \rotatebox[origin=c]{270}{$\subseteq$} & &\\
 & & \widehat{X}_{\mathtt{CDCP}} & \subseteq & \widehat{X}_{\mathtt{DCP}}.
\end{array} \notag
\end{empheq}
}

Note that $\Xsa{\delta_1}$ and $\Xscp{0}$ are in general incomparable with $\Xcdcp$. Thus, the objective values obtained by optimizing over these sets are not necessarily upper or lower bounds on the optimal solution of \eqref{eq:def_drccp_approx_reform}.

We conclude this section with the following {\it ex-post} comparison of the feasibility sets $\Xscp{0}$ and $\Xcdcpin$. 

\begin{proposition}\longthmtitle{Ex-post comparison of $\Xcdcpin$ and $\Xscp{0}$}\label{prop:scenario_dcp_feasibility}
	Let Assumptions \ref{ass:main1} and \ref{assumption:f_convex} hold. Assume $L_F$ is constant over $X$. Let $x \in \Xscp{0}$. Define $J_x := \{i \in [N]|F(x,\widehat{\xi}_i) = 0\}$ and $\gamma_x := \underset{i \in [N] \setminus J_x}{\min} (-F(x,\widehat{\xi}_i))$. If $\theta \leq \frac{\gamma_x}{L_F}(\alpha-\frac{|J_x|}{N})$, then $x \in \Xcdcpin$.
\end{proposition}
\begin{proof}
	Let $t = \gamma_x$. Then, for $i \in [N] \setminus J_x$, $F(x,\data_i) + \gamma \leq 0$. Therefore,
	\begin{align*}
	& L_F\theta+ \frac{1}{N} \sum_{i=1}^{N} (F(\bar{x},\data_i)+\gamma_x)_+ -\gamma_x \alpha 
	\\ & \quad = L_F\theta - \gamma_x \alpha + \frac{1}{N} \sum_{i \in J_x} \gamma_x = L_F\theta - \gamma_x \alpha + \frac{|J_x|}{N}\gamma_x \leq 0.
	\end{align*}
	Thus, we deduce that $x \in \widehat{X}^{\mathtt{in}}_{\mathtt{CDCP}}$.
\end{proof}

As a consequence of the above result, for a given optimal solution $\xo \in X$ of the scenario program, if $\frac{\gamma_{\xo}}{L_F}(\alpha-\frac{|J_{\xo}|}{N}) > 0$, we can choose the radius of the Wasserstein ambiguity set $\theta$ to be sufficiently small such that the optimal solution of the DRCCP with the feasibility set $\widehat{X}^{\mathtt{in}}_{\mathtt{CDCP}}$ has a smaller value compared to the scenario program. 

\section{Conclusion}\label{sec:conclusions}

We studied distributionally robust chance constrained optimization under Wasserstein ambiguity sets defined as the set of all distributions that are close to the empirical distribution. We presented a convex reformulation of the program when the original chance constraint is replaced by its convex CVaR counterpart. We then showed the tractability of this convex reformulation for affine constraint functions. Furthermore, for constraint functions concave in the uncertainty, we presented a cutting-surface algorithm that converges to an approximately optimal solution of the CVaR approximation of the DRCCP. Finally, for constraint functions convex in the uncertainty, we compared the feasibility sets of DRCCP and its approximations with those of the scenario and sample approximation approaches. 

In future, we plan to build upon our results to design distributionally robust controllers for stochastic systems. In addition, we wish explore online optimization approaches for DRCCPs, and investigate their relevance for stochastic model predictive control problems. A rigorous comparison of DRCCPs and the scenario approach vis-a-vis finite sample guarantees and asymptotic convergence of optimal solutions also remain as challenging open problems.

\section*{Acknowledgement}
We thank Prof. Shabbir Ahmed and Prof. Weijun Xie for pointing out an error in an earlier version of this work as well as pointers to certain relevant papers. We thank Prof. Peyman Mohajerin Esfahani for helpful discussions. 

\bibliographystyle{IEEEtran}
\bibliography{bibfile}

\begin{thebibliography}{10}
\providecommand{\url}[1]{#1}
\csname url@rmstyle\endcsname
\providecommand{\newblock}{\relax}
\providecommand{\bibinfo}[2]{#2}
\providecommand\BIBentrySTDinterwordspacing{\spaceskip=0pt\relax}
\providecommand\BIBentryALTinterwordstretchfactor{4}
\providecommand\BIBentryALTinterwordspacing{\spaceskip=\fontdimen2\font plus
\BIBentryALTinterwordstretchfactor\fontdimen3\font minus
  \fontdimen4\font\relax}
\providecommand\BIBforeignlanguage[2]{{%
\expandafter\ifx\csname l@#1\endcsname\relax
\typeout{** WARNING: IEEEtran.bst: No hyphenation pattern has been}%
\typeout{** loaded for the language `#1'. Using the pattern for}%
\typeout{** the default language instead.}%
\else
\language=\csname l@#1\endcsname
\fi
#2}}

\bibitem{ben2009robust}
A.~Ben-Tal, L.~El~Ghaoui, and A.~Nemirovski, \emph{Robust optimization}.\hskip
  1em plus 0.5em minus 0.4em\relax Princeton University Press, 2009, vol.~28.

\bibitem{shapiro2009lectures}
A.~Shapiro, D.~Dentcheva, and A.~Ruszczy{\'n}ski, \emph{Lectures on stochastic
  programming: {M}odeling and theory}.\hskip 1em plus 0.5em minus 0.4em\relax
  SIAM, 2009.

\bibitem{farina2016stochastic}
M.~Farina, L.~Giulioni, and R.~Scattolini, ``Stochastic linear model predictive
  control with chance constraints--a review,'' \emph{Journal of Process
  Control}, vol.~44, pp. 53--67, 2016.

\bibitem{schildbach2014scenario}
G.~Schildbach, L.~Fagiano, C.~Frei, and M.~Morari, ``The scenario approach for
  stochastic model predictive control with bounds on closed-loop constraint
  violations,'' \emph{Automatica}, vol.~50, no.~12, pp. 3009--3018, 2014.

\bibitem{blackmore2011chance}
L.~Blackmore, M.~Ono, and B.~C. Williams, ``Chance-constrained optimal path
  planning with obstacles,'' \emph{IEEE Transactions on Robotics}, vol.~27,
  no.~6, pp. 1080--1094, 2011.

\bibitem{vitus2016stochastic}
M.~P. Vitus, Z.~Zhou, and C.~J. Tomlin, ``Stochastic control with uncertain
  parameters via chance constrained control,'' \emph{IEEE Transactions on
  Automatic Control}, vol.~61, no.~10, pp. 2892--2905, 2016.

\bibitem{vrakopoulou2017chance}
M.~Vrakopoulou, B.~Li, and J.~L. Mathieu, ``Chance constrained reserve
  scheduling using uncertain controllable loads part {I}: Formulation and
  scenario-based analysis,'' \emph{IEEE Transactions on Smart Grid (To
  appear)}, 2017.

\bibitem{guo2018data}
Y.~Guo, K.~Baker, E.~Dall'Anese, Z.~Hu, and T.~H. Summers, ``Data-based
  distributionally robust stochastic optimal power flow, part {I}:
  Methodologies,'' \emph{arXiv preprint arXiv:1804.06388}, 2018.

\bibitem{carvalho2015automated}
A.~Carvalho, S.~Lef{\'e}vre, G.~Schildbach, J.~Kong, and F.~Borrelli,
  ``Automated driving: The role of forecasts and uncertainty - {A} control
  perspective,'' \emph{European Journal of Control}, vol.~24, pp. 14--32, 2015.

\bibitem{calafiore2005uncertain}
G.~C. Calafiore and M.~C. Campi, ``Uncertain convex programs: {R}andomized
  solutions and confidence levels,'' \emph{Mathematical Programming}, vol. 102,
  no.~1, pp. 25--46, 2005.

\bibitem{campi2008betterbound}
M.~C. Campi and S.~Garatti, ``The exact feasibility of randomized solutions of
  uncertain convex programs,'' \emph{SIAM Journal on Optimization}, vol.~19,
  no.~3, pp. 1211--1230, 2008.

\bibitem{calafiore2010random}
G.~C. Calafiore, ``Random convex programs,'' \emph{SIAM Journal on
  Optimization}, vol.~20, no.~6, pp. 3427--3464, 2010.

\bibitem{luedtke2008sample}
J.~Luedtke and S.~Ahmed, ``A sample approximation approach for optimization
  with probabilistic constraints,'' \emph{SIAM Journal on Optimization},
  vol.~19, no.~2, pp. 674--699, 2008.

\bibitem{campi2011samplediscard}
M.~C. Campi and S.~Garatti, ``A sampling-and-discarding approach to
  chance-constrained optimization: feasibility and optimality,'' \emph{Journal
  of Optimization Theory and Applications}, vol. 148, pp. 257--280, 2011.

\bibitem{delage2010distributionally}
E.~Delage and Y.~Ye, ``Distributionally robust optimization under moment
  uncertainty with application to data-driven problems,'' \emph{Operations
  research}, vol.~58, no.~3, pp. 595--612, 2010.

\bibitem{zymler2013joint}
S.~Zymler, D.~Kuhn, and B.~Rustem, ``Distributionally robust joint chance
  constraints with second-order moment information,'' \emph{Mathematical
  Programming}, vol. 137, pp. 167--198, 2013.

\bibitem{hanasusanto2017ambiguous}
G.~A. Hanasusanto, V.~Roitch, D.~Kuhn, and W.~Wiesemann, ``Ambiguous joint
  chance constraints under mean and dispersion information,'' \emph{Operations
  Research}, vol.~65, no.~3, pp. 751--767, 2017.

\bibitem{erdougan2006ambiguous}
E.~Erdo{\u{g}}an and G.~Iyengar, ``Ambiguous chance constrained problems and
  robust optimization,'' \emph{Mathematical Programming}, vol. 107, no. 1-2,
  pp. 37--61, 2006.

\bibitem{jiang2016data-driven}
R.~Jiang and Y.~Guan, ``Data-driven chance constrained stochastic program,''
  \emph{Mathematical Programming}, vol. 158, pp. 291--327, 2016.

\bibitem{van2016distributionally}
B.~P. Van~Parys, D.~Kuhn, P.~J. Goulart, and M.~Morari, ``Distributionally
  robust control of constrained stochastic systems,'' \emph{IEEE Transactions
  on Automatic Control}, vol.~61, no.~2, pp. 430--442, 2016.

\bibitem{zhang2017distributionally}
Y.~Zhang, S.~Shen, and J.~L. Mathieu, ``Distributionally robust
  chance-constrained optimal power flow with uncertain renewables and uncertain
  reserves provided by loads,'' \emph{IEEE Transactions on Power Systems},
  vol.~32, no.~2, pp. 1378--1388, 2017.

\bibitem{villani2003topics}
C.~Villani, \emph{Topics in optimal transportation}.\hskip 1em plus 0.5em minus
  0.4em\relax American Mathematical Soc., 2003, no.~58.

\bibitem{gao2016wasserstein}
R.~Gao and A.~J. Kleywegt, ``Distributionally robust stochastic optimization
  with {W}asserstein distance,'' 2016, available online at
  \url{https://arxiv.org/abs/1604.02199}.

\bibitem{peyman2017wasserstein}
P.~M. Esfahani and D.~Kuhn, ``Data-driven distributionally robust optimization
  using the {W}asserstein metric: Performance guarantees and tractable
  reformulations,'' \emph{Mathematical Programming. Ser. A}, 2017, to appear.
  Available at \url{https://arxiv.org/abs/1505.05116}.

\bibitem{yang2017convex}
I.~Yang, ``A convex optimization approach to distributionally robust markov
  decision processes with {W}asserstein distance,'' \emph{IEEE control systems
  letters}, vol.~1, no.~1, pp. 164--169, 2017.

\bibitem{ahmed2018bicriteria}
W.~Xie and S.~Ahmed, ``Bicriteria approximation of chance constrained covering
  problems,'' 2018, available online at
  \url{http://www.optimization-online.org/DB_HTML/2018/01/6411.html}.

\bibitem{xie2018drccp}
W.~Xie, ``On distributionally robust chance constrained program with
  {W}asserstein distance,'' 2018, available online at
  \url{http://www.optimization-online.org/DB_FILE/2018/06/6662.pdf}.

\bibitem{kuhn2018drccp}
Z.~Chen, D.~Kuhn, and W.~Wiesemann, ``Data-driven chance constrained programs
  over {W}asserstein balls,'' 2018, available online at
  \url{http://www.optimization-online.org/DB_FILE/2018/06/6671.pdf}.

\bibitem{prekopa1970probabilistic}
A.~Prekopa, ``On probabilistic constrained programming,'' in \emph{Proceedings
  of the Princeton symposium on mathematical programming}, vol. 113, 1970, p.
  138.

\bibitem{nemirovski2006convex}
A.~Nemirovski and A.~Shapiro, ``Convex approximations of chance constrained
  programs,'' \emph{SIAM Journal on Optimization}, vol.~17, no.~4, pp.
  969--996, 2006.

\bibitem{mehrotra2014cutting}
S.~Mehrotra and D.~Papp, ``A cutting surface algorithm for semi-infinite convex
  programming with an application to moment robust optimization,'' \emph{SIAM
  Journal on Optimization}, vol.~24, no.~4, pp. 1670--1697, 2014.

\bibitem{luo2017decomposition}
F.~Luo and S.~Mehrotra, ``Decomposition algorithm for distributionally robust
  optimization using wasserstein metric,'' 2017, arXiv preprint
  arXiv:1704.03920.

\bibitem{rockafellar2000optimization}
R.~T. Rockafellar and S.~Uryasev, ``Optimization of conditional
  value-at-risk,'' \emph{Journal of risk}, vol.~2, pp. 21--42, 2000.

\bibitem{shapiro2002minimax}
A.~Shapiro and A.~Kleywegt, ``Minimax analysis of stochastic problems,''
  \emph{Optimization Methods and Software}, vol.~17, no.~3, pp. 523--542, 2002.

\bibitem{pichler2017quantitative}
A.~Pichler and H.~Xu, ``Quantitative stability analysis for minimax
  distributionally robust risk optimization,'' 2017, available online at
  \url{http://www.optimization-online.org/DB_HTML/2017/01/5814.html}.

\bibitem{rockafellar1970convex-analysis}
R.~T. Rockafellar, \emph{Convex Analysis}.\hskip 1em plus 0.5em minus
  0.4em\relax Princeton, NJ: Princeton University Press, 1970.

\end{thebibliography}
\end{document}